\documentclass[12pt,a4paper]{amsart}

\usepackage{amssymb}
\usepackage[pdftex]{graphicx}
\usepackage{amssymb,amsfonts,epsfig,amsthm,a4,amsmath,url}
\usepackage[latin1]{inputenc}

\newtheorem{thmm}{Theorem}
\newtheorem{thm}{Theorem}[section]
\newtheorem{cor}[thm]{Corollary}
\newtheorem{lem}[thm]{Lemma}
\newtheorem{prop}[thm]{Proposition}

\theoremstyle{definition}
\newtheorem{cla}{Claim}
\newtheorem{exe}[thm]{Example}
\newtheorem{rem}[thm]{Remark}

\numberwithin{equation}{section}

\newcommand{\SOL}{\textnormal{SOL}}
\newcommand{\GL}{\textnormal{GL}}
\newcommand{\BS}{\textnormal{BS}}
\newcommand{\eps}{\varepsilon}
\newcommand{\lp}{(\!(}
\newcommand{\rp}{)\!)}

\begin{document}

\title[Metabelian groups with quadratic Dehn function]{Metabelian groups with quadratic Dehn function and Baumslag-Solitar groups}

\author{Yves de Cornulier}
\address{IRMAR \\ Campus de Beaulieu \\
35042 Rennes Cedex, France}
\email{yves.decornulier@univ-rennes1.fr}

\author{Romain Tessera}
\address{UMPA, ENS Lyon\\46 All\'ee d'Italie\\ 69364 Lyon Cedex\\ France}
\email{rtessera@umpa.ens-lyon.fr}


		
\date{November 26, 2010}
\maketitle

\begin{abstract}
We prove that a certain class of metabelian locally compact groups have quadratic Dehn function. As an application, we embed the solvable Baumslag-Solitar groups in finitely presented metabelian groups with quadratic Dehn function. Also, we prove that Baumslag's finitely presented metabelian groups, in which the lamplighter groups embed, have quadratic Dehn function. 
\end{abstract}


\section{Introduction}

The {\it Dehn function} of a finitely presented group is a natural invariant, which from the combinatorial point of view is a measure of the complexity of the word problem, and from the geometric point of view can be interpreted as an isoperimetry invariant. We refer the reader to Bridson's survey \cite{Bridson} for a detailed discussion. Given a group with a solvable word problem, it is natural to wonder whether it can be embedded into a group with small Dehn function; a major result in this context is a characterization of subgroups of groups with at most polynomial Dehn function as those groups with NP solvable word problem \cite{BORS}. It is then natural to ask for which groups this can be improved. 
Word hyperbolic groups are characterized as those with linear Dehn function and there are many known stringent restrictions on their subgroups; for instance, their abelian subgroups are virtually cyclic. In particular, they cannot contain a solvable Baumslag-Solitar group 
$$\BS(1,n)=\langle t,x\,|\;txt^{-1}=x^n\rangle$$
for any integer $n$ with $|n|\ge 1$. The next question is whether the group is embeddable into a finitely presented group with quadratic Dehn function. We obtain here a positive answer for the solvable Baumslag-Solitar groups.
The group $\BS(1,n)$ is well-known to have an exponential Dehn function whenever $|n|\ge 2$, as for instance it is proved in \cite{GH} that any strictly ascending HNN-extension of a finitely generated nilpotent group has an exponential Dehn function.

\begin{thmm}\label{main}
The solvable Baumslag-Solitar group $\BS(1,n)$ can be embedded into a finitely presented metabelian group with quadratic Dehn function.
\end{thmm}

This answers a question asked in \cite{BORS}.
It was previously known, by the general result mentioned above, that it can be embedded into a ``big" finitely presented group with polynomial Dehn function \cite{BORS}. Then $\BS(1,n)$ was proved to be embeddable in a finitely presented metabelian group with at most cubic Dehn function in \cite{AO}. Theorem \ref{main} is optimal, because as we mentioned above, $\BS(1,n)$ cannot be embedded into a word hyperbolic group, and all other groups have a quadratic lower bound on their Dehn function \cite{Ols}. Quadratic Dehn function groups enjoy some special properties not shared by all polynomial Dehn function groups, for instance they have all asymptotic cones simply connected \cite{Papa}. After Gromov \cite{Gromov}, many solvable groups were proved to have quadratic Dehn function (see for instance \cite{allcock,Drutu,Pittet}).

Our construction of a finitely presented metabelian group with quadratic Dehn function in which $\BS(1,n)$ embeds, is elementary and uses representation by matrices; the proof involves an embedding as a cocompact lattice into a group of matrices over a product of local fields. The theorem then follows from a general result (Theorem \ref{meta}), allowing to prove that various metabelian groups, made up from local fields, have quadratic Dehn function. Another application concerns a finitely presented metabelian group $\Lambda_p$, introduced by G.~Baumslag \cite{Bau} as an instance of a finitely presented metabelian group in which the lamplighter group $(\mathbf{Z}/p\mathbf{Z}) \wr \mathbf{Z}$ embeds as a subgroup (see Example \ref{Baumslag group}).
\begin{thmm}\label{main'}
Baumslag's finitely presented metabelian group $\Lambda_p$ has quadratic Dehn function.
\end{thmm}
A polynomial upper bound was recently obtained by Kassabov and Riley \cite{KR}.

\medskip

We hope that the method developed in the paper will convince the reader that the study of the Dehn function of discrete groups is indissociable from its study for general locally compact groups, a fact more generally true in the problem of quasi-isometry classification of solvable groups.

\subsection*{Organisation}
In the next section, we describe our embedding of the Baumslag-Solitar group. In Section~\ref{metaSection}, we state our fundamental result, namely Theorem~\ref{meta}. Section \ref{TrickSection} is dedicated to an important technical lemma whose basic idea is mainly due to Gromov. Finally, Section \ref{ProofSection} contains the proof of Theorem \ref{meta}.

\medskip
\noindent \textbf{Acknowledgement.} We thank Mark Sapir for suggesting us this problem and for valuable discussions. We are indebted to the referee for many corrections and clarifications.


\section{Construction of the embedding}

Our results use, in an essential way, a notion of Dehn function not restricted to discrete groups.
Let $G$ be any locally compact group. Recall \cite[Section~1.1]{A} that $G$ is {\it compactly presented} if for some/any compact generating symmetric set $S$ and some $k$ (depending on $S$), there exists a presentation of the abstract group $G$ with $S$ as set of generators, and relators of length $\le k$. If $G$ is discrete, this amounts to say that $G$ is finitely presented. If $(S,k)$ is fixed, a relation means a word $w$ in the letters of $S$ which represents the trivial element of $G$; its {\it area} is the least $m$ such that $w$ is a product, in the free group, of $\le m$ conjugates of relations of length $\le k$. The {\it Dehn function} of $G$ is defined as
$$\delta(n)=\sup\{\text{area}(w)|w \text{ relation of length }\le n\}.$$
The precise value of $\delta(n)$ depends on $(S,k)$, but not the $\approx$-asymptotic behavior, where $u(n)\approx v(n)$ if for suitable positive constants $a_1,\dots,b_4$ independent of $n$
$$a_1u(a_2n)-a_3n-a_4\le v(n)\le b_1u(b_2n)+b_3n+b_4,\quad\forall n\ge 0.$$
In the combinatorial point of view, the Dehn function of a finitely presented group is a measure of the complexity of its word problem. In the geometric point of view, compact presentability means simple connectedness at large scale \cite[1.$\textnormal{C}_1$]{Gromov}, and the Dehn function appears as a quantified version, and an upper bound on the Dehn function is often referred as an ``isoperimetric inequality".

For any non-zero integer $n\in\mathbf{Z}-\{0\}$, the solvable Baumslag-Solitar group can be described as $$\BS(1,n)=\mathbf{Z}[1/n]\rtimes\mathbf{Z},$$
where $\mathbf{Z}$ acts on $\mathbf{Z}[1/n]$ by multiplication by $n$.

Consider the two commuting matrices $A=\begin{pmatrix}
  n & 0 \\
  0 & n \\
\end{pmatrix}$, $B=\begin{pmatrix}
  2 & 1 \\
  1 & 1 \\
\end{pmatrix}$ ($B$ can be replaced by any matrix in $\GL_2(\mathbf{Z})$ with two real eigenvalues not of modulus one). Define the group $$\Gamma_n=\mathbf{Z}[1/n]^2\rtimes_{(A,B)}\mathbf{Z}^2.$$
Clearly, $\Gamma_n$ is finitely generated, since we can ``go up" in $\mathbf{Z}[1/n]^2$ by conjugating by $A$. Moreover, it contains an obvious copy of $\BS(1,n)$, namely $(\mathbf{Z}[1/n]\times\{0\})\rtimes(\mathbf{Z}\times\{0\})$.

\begin{thm}\label{gaga}
The group $\Gamma_n$ is finitely presented with quadratic Dehn function.
\end{thm}

\begin{cor}
The group $\BS(1,n)$ can be embedded into a finitely presented group with quadratic Dehn function.
\end{cor}

Theorem \ref{gaga} is obtained by working inside a more convenient group, which contains $\Gamma_n$ as a cocompact lattice.
Let $\mathbf{Q}_p$ denote the $p$-adic field, and define the ring $\mathbf{Q}_n$ as the direct product of all $\mathbf{Q}_p$ when $p$ ranges over the {\it set of distinct} prime divisors of $n$. Then the natural diagonal embedding of $\mathbf{Z}[1/n]$ into $\mathbf{Q}_n\oplus\mathbf{R}$ has discrete cocompact image (check it as an exercise or see \cite[Chap.~IV, \S 2]{Weil}), and the corresponding embedding of $\mathbf{Z}[1/n]^k$ into $\mathbf{Q}_n^k\oplus\mathbf{R}^k$ is equivariant for the natural actions of $\GL_k(\mathbf{Z}[1/n])$.
Accordingly, $\Gamma_n$ stands as a cocompact lattice in the locally compact group
$$(\mathbf{Q}_n^2\oplus\mathbf{R}^2)\rtimes\mathbf{Z}^2,$$
where the action is still defined by the same pair of matrices $(A,B)$ (viewed as matrices over the ring $\mathbf{Q}_n\times\mathbf{R}$). Now we use the fact that the (asymptotic behavior of the) Dehn function is a quasi-isometry invariant (see \cite{Alonso} for a proof in the discrete setting; the same proof working in the general case), so $\Gamma_n$ has the same Dehn function as this larger group.

Decompose $\mathbf{R}^2=V_+\oplus V_-$ along the eigenspaces of $B$, so that $B$ dilates $V_+$ and contracts $V_-$. Observe that 
\begin{itemize}
\item $BA^{-1}$ contracts both $V_-$ and $\mathbf{Q}_n^2$ 
\item $B^{-1}A^{-1}$ contracts both $V_+$ and $\mathbf{Q}_n^2$ 
\item $A^{-1}$ contracts both $V_+$ and $V_-$.
\end{itemize}

Thus, for every pair among $\mathbf{Q}_n^2$, $V_+$ and $V_-$, there is a common contraction of the acting group. This phenomenon is enough to prove that the Dehn function is quadratic, and Theorem \ref{gaga} is a consequence of a more general result, which is the object of the next section (and will be proved in the last section).

\begin{rem}
It was observed in \cite[Section~9]{C} that the asymptotic cone of $\BS(1,n)$ ($|n|\ge 2$) or $\SOL_3(\mathbf{K})$ is (for any choice of an ultrafilter) bilipschitz homeomorphic to the Diestel-Leader $\mathbf{R}$-graph $$\{(x,y)\in\mathbf{T}\times\mathbf{T}:b(x)+b(y)=0\},$$ where $\mathbf{T}$ is the universal complete $\mathbf{R}$-tree everywhere branched of degree $2^{\aleph_0}$ and $b$ any Busemann function on $\mathbf{T}$. The same reasoning shows that any asymptotic cone of $\Gamma_n$ is bilipschitz homeomorphic to $$\{(x,y,z)\in\mathbf{T}^3:b(x)+b(y)+b(z)=0\}$$
(which is also bilipschitz homeomorphic to the asymptotic cone of the $\SOL_5(\mathbf{K})$ for any local field $\mathbf{K}$, as follows from the final remarks in \cite[Section~9]{C}). 
\end{rem}


\section{A general result and further comments}\label{metaSection}

By {\it local field} we mean a non-discrete locally compact normed field, see \cite{Weil}.
Consider a semidirect product 
$$G=\bigoplus_{i=1}^m V_i\rtimes A,$$
where $A$ is a finitely generated abelian group of rank $d$ and where $V_i$ is $A$-invariant and is isomorphic to a product of local fields, on which $A$ acts by scalar multiplication.

\begin{thm}\label{meta}
 Assume that for every pair $i,j$, there exists an element of $A$ acting by strict contractions on both $V_i$ and $V_j$ (i.e.\ acting by multiplication by an element of norm $<1$ on each factor). Then $G$ is compactly presented with quadratic Dehn function if $d\ge 2$, and linear Dehn function if $d=1$.
\end{thm}

\begin{rem}
The consideration of common contractions is ubiquitous in this context, see for instance \cite{BS,A}.
\end{rem}

\begin{exe}\label{Baumslag group}
Let $\mathbf{K}$ be any local field. Let $\SOL_{2d-1}(\mathbf{K})$ be the semidirect product of $\mathbf{K}^d$ by the set of diagonal matrices with determinant of norm one. It has a cocompact subgroup of the form $\mathbf{K}^d\rtimes\mathbf{Z}^{d-1}$, obtained explicitly by reducing to matrices whose diagonal entries are all powers of a given element of $\mathbf{K}$ of norm $\neq 1$, so we can apply Theorem \ref{meta}, which yields that $\SOL_{2d-1}(\mathbf{K})$ has quadratic Dehn function whenever $2d-1\ge 5$. This was proved by Gromov when $\mathbf{K}=\mathbf{R}$ \cite[5.$\textnormal{A}_9$]{Gromov}. As a further application, let $p$ be prime and consider the finitely presented group
$$\Lambda_p=\langle a,s,t\,|\;\; a^p,\;[s,t],\;[a^t,a],\;a^s=a^ta\rangle,$$ which was introduced by Baumslag \cite{Bau} as a finitely presented metabelian group containing a copy of the lamplighter group $(\mathbf{Z}/p\mathbf{Z})\wr\mathbf{Z}$. 
Let $\mathbf{F}_p\lp t\rp$ denote the field of Laurent series over the finite field $\mathbf{F}_p$. The following proposition is very standard, but we have no reference for a complete proof. As a consequence we deduce that $\Lambda_p$ has quadratic Dehn function.
\begin{prop}\label{latp}
For any prime $p$, the group $\Lambda_p$
embeds as a cocompact lattice into $\SOL_5(\mathbf{F}_p\lp u\rp)$.
\end{prop}
\begin{proof}[Proof (sketched)]
Let $\mathbf{F}_p[X,X^{-1}]$ be the ring of Laurent polynomials over $\mathbf{F}_p$. Consider the group 
$$\Omega_p=\mathbf{F}_p[X,X^{-1},(1+X)^{-1}]\rtimes\mathbf{Z}^2,$$
where the generators of $\mathbf{Z}^2$ act by multiplication by $X$ and $1+X$ respectively. 
We have an obvious homomorphism $\Lambda_p\to\Omega_p$ mapping $a$ to the unit element of the ring $\mathbf{F}_p[X,X^{-1},(1+X)^{-1}]$ and $(t,s)$ to the canonical basis of $\mathbf{Z}^2$; it is essentially contained in Baumslag's proof \cite{Bau} that this is an isomorphism.

Consider the embedding
\begin{eqnarray*}\sigma:\mathbf{F}_p[X,X^{-1},(1+X)^{-1}]\to\mathbf{F}_p\lp u\rp^3\\
(P_1,P_2,P_3)\mapsto (P_1(u),P_2(u^{-1}),P_3(u-1))
\end{eqnarray*}
This is an embedding as a cocompact lattice: this can be checked by hand (see \cite[Proof of Prop.~3.4]{LSV}), or follows from general results~\cite{Ha}.

If we make the two generators of $\mathbf{Z}^2$ act on $\mathbf{F}_p\lp u\rp^3$ by the diagonal matrices
$$(u,u^{-1},u-1)\quad\text{and}\quad(u+1,u^{-1}+1,u),$$
this makes $\sigma$ a $\mathbf{Z}^2$-equivariant embedding. Moreover, denoting by $D^1_3(\mathbf{K})$ the group of $3\times 3$ matrices with determinant of norm one, it is readily seen that this embedding of $\mathbf{Z}^2$ into $D^1_3(\mathbf{F}_p\lp u\rp)$ has discrete and cocompact image. Therefore the embedding $\sigma$ extends to a discrete cocompact embedding of $\Omega_p$ into 
$$\mathbf{F}_p\lp u\rp^3\rtimes D^1_3(\mathbf{K})=\SOL_5(\mathbf{F}_p\lp u\rp).\qedhere$$
\end{proof}

\begin{rem}
If $n\ge 2$ is not prime, replacing, in the argument, $\mathbf{F}_p\lp u\rp$ by $(\mathbf{Z}/n\mathbf{Z})\lp u\rp$, we readily obtain that $\Lambda_n$ has quadratic Dehn function as well. However, the argument does not apply for $n=0$, because $\Lambda_0$ contains a copy of the wreath product $\mathbf{Z}\wr\mathbf{Z}$ and therefore does not stand as a discrete linear group over a product of local fields, and actually Kassabov and Riley \cite{KR} proved that $\Lambda_0$ has exponential Dehn function. 
\end{rem}
\end{exe}

\begin{rem}It is natural to ask whether there is a group with some embedding into a discrete group with polynomial Dehn function, but not into one with quadratic Dehn function. The answer is positive, as M.~Sapir showed to the authors: consider a problem in {\sc ntime}($n^3$) (that is, solvable in cubic time by a non-deterministic Turing machine), but not in {\sc ntime}($n^2$). Such problems exist by \cite[p.~76]{hcs} or \cite[p.~69-70]{arba}. By \cite{BORS}, there exists a finitely presented group $\Gamma$ with word problem in {\sc ntime}($n^3$) but not {\sc ntime}($n^2$), and again by \cite{BORS} there exists a finitely presented group $\Lambda$ with polynomial Dehn function containing $\Gamma$ as a subgroup. However, $\Gamma$ cannot be embedded into a finitely presented group with quadratic Dehn function, since otherwise, using \cite[Theorem~1.1]{Sab} its word problem would be in {\sc ntime}($n^2$). Still, it would be interesting to have an example of more geometrical nature.\end{rem}


\section{Reduction to special words}\label{TrickSection}

In this section, we prove Proposition \ref{quadra}, which reduces the computation of the Dehn function to its computation for words of a special form. The basic idea is due to Gromov \cite[p.~86]{Gromov}.

\begin{lem}\label{lemf}
Let $u_k:\mathbf{R}_{>0}\to\mathbf{R}_{>0}$ be a family of functions, indexed by integers $k\ge 1$, satisfying $\underline{\lim}_{x\to \infty} u_k(x)> 0$ for all $k$. Assume that
$$\frac{zu_k(y)}{u_k(yz)}\underset{z\to\infty}{\longrightarrow}0\;\text{uniformly in }y\ge 1,k\ge 1.$$

Consider a function $f:\mathbf{R}_{>0}\to\mathbf{R}_{>0}$, locally bounded
and positive constants $c_1,c_2,x_0$, such that for all $k\ge 1,x\ge x_0$
$$f(x)\le u_k(x)+c_1kf(c_2x/k).$$
Then, for some constants $A,x_1$ and some $k\ge 1$, we have $$f(x)\le Au_k(x),\quad\forall x\ge x_1.$$
\end{lem}

For example, $u_k(x)=a_kx^\alpha$, for $\alpha>1$ and arbitrary constants $a_k>0$, satisfy the assumption.

\begin{proof}
Set $\eta=1/(2c_1c_2)$. There exists, by the assumption, $\eps_0>0$ (we choose $\eps_0\le 1/2$) such that for all $z\ge 1/\eps_0$, $y,k\ge 1$ we have $\frac{zu_k(y)}{u_k(yz)}\le\eta$. Therefore, for all $x,\eps>0$ such that $x\eps\ge 1$ and $\eps\le\eps_0$ we have $\frac{u_k(x\eps)}{\eps u_k(x)}\le\eta$ (as we check by setting $y=x\eps$ and $z=\eps^{-1}$).
Taking $\eps=c_2/k$, we get, for $k\ge c_2/\eps_0$ and for all $x\ge\eps^{-1}$
$$u_k\left(\frac{c_2x}{k}\right)\le\frac{c_2\eta}{k}u_k(x).$$
We now fix $k\ge c_2/\eps_0$ (so $\eps$ is fixed as well and $\eps\le\eps_0\le 1/2$).
We let $x_1\geq \max(\eps^{-1},x_0k)$ be large enough so that $\inf_{x\ge x_1}u_k(x)>0$.
Therefore, since $f$ is locally bounded, there exists $A\ge 2$ such that for all $x\in [x_1,\eps^{-1}x_1]$ we have $f(x)\le Au_k(x)$. 

Now let us prove that $f(x)\le Au_k(x)$ for all $x\ge x_1$, showing by induction on $n\ge 1$ that 
$f(x)\le Au_k(x)$ for all $x\in [x_1,\eps^{-n}x_1]$. It already holds for $n=1$; suppose that the induction is proved until $n-1\ge 1$. Take $x\in [\eps^{1-n}x_1,\eps^{-n}x_1]$. By induction hypothesis, we have $f(x')\le Au_k(x')$ for all $x'\in[x_1,\eps x]$.
We have
\begin{align*}
f(x) & \le u_k(x)  +c_1kf(c_2x/k)\\
 & \le u_k(x)+c_1kAu_k(c_2x/k)\\
& \le u_k(x)(1 + c_1Ac_2\eta)\\ &\le u_k(x)(1+A/2)\quad\le Au_k(x).\qedhere
 \end{align*}
\end{proof}

If $x$ is real, we denote by $\lfloor x\rfloor=\sup (]-\infty,x]\cap\mathbf{Z})$ and $\lceil x\rceil=\inf([x,+\infty[\cap\mathbf{Z})$ its lower and upper integer parts. 

Let $G$ be a locally compact group generated by a compact symmetric subset $S$. Let $F_S$ be the nonabelian free group over $S$. For $w,w'\in F_S$, we write $w\equiv w'$ if $w$ and $w'$ represent the same element of $G$.

Let $\mathcal{F}$ be a set of words in $S$. We write $\mathcal{F}[k]$ the set of words obtained as the concatenation of $\le k$ words of $\mathcal{F}$. We can define the restricted Dehn function
$\delta_\mathcal{F}(n)$ as the supremum of areas of null-homotopic words in $\mathcal{F}$ (say $\sup\emptyset$=0).
We say that $\mathcal{F}$ is {\it efficient} if there exists a constant $C$ such that for every word $w$ in $S$, there exists $w'\in\mathcal{F}$ such that $w'\equiv w$ and $|w'|_S\le C|w|_S$. Also, if $x$ is a nonnegative real number, $\delta(x)$ can obviously be defined as the supremum of areas of loops of length $\le x$ (so $\delta(x)=\delta(\lfloor x\rfloor)$).

\begin{lem}\label{astg}
Suppose that $\mathcal{F}$ is efficient. Then for any $k\in\mathbf{Z}_{>0}$ and $n\in\mathbf{R}_{>0}$ we have
$$\delta(n)\le k\delta\left((C+1)\left\lceil \frac nk\right\rceil\right)+\delta_{\mathcal{F}[k]}\left(Ck\left\lceil \frac nk\right\rceil\right);$$
in particular for $n\ge k$ we have
$$\delta(n)\le k\delta\left((C+1)\frac {2n}k\right)+\delta_{\mathcal{F}[k]}(2Cn).$$
\end{lem}
\begin{proof}
Consider a loop $\gamma$ of length $\le n$ and cut it into $k$ segments $[a_i,a_{i+1}]$ of length $\le \lceil n/k\rceil$ (see Figure \ref{fig:theFig}). Set $b_i=a_i^{-1}a_{i+1}$; there exists $b'_i\equiv b_i$ with $b'_i\in\mathcal{F}$ and $b'_i\le C\lceil n/k\rceil$. Set $\gamma_i=b'_ib_i^{-1}$, so $\gamma_i$ is null-homotopic.

\begin{figure}[!t]
\includegraphics[scale=0.5]{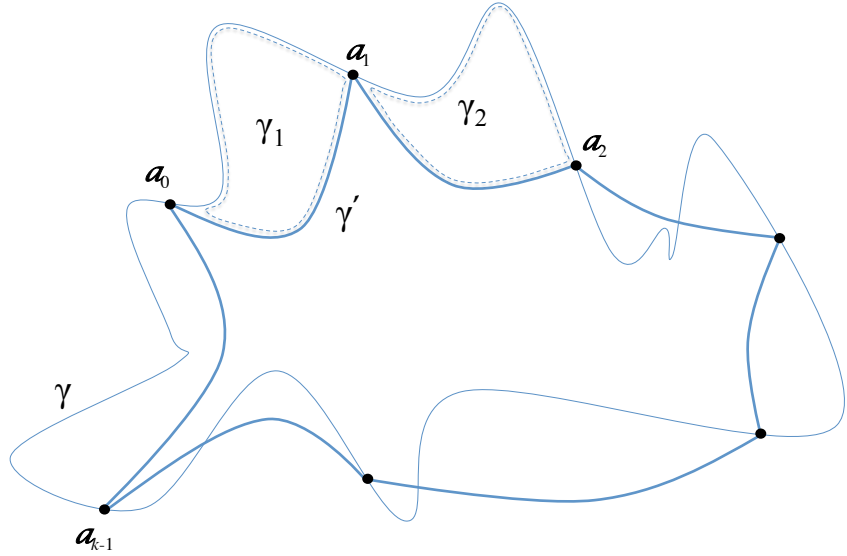}
\caption{The path $\gamma$ cut into $k$ segments.}
\label{fig:theFig}
\end{figure}

Thus the loop $\gamma$ has been decomposed into $k$ loops $\gamma_1,\dots,\gamma_k$ of length $\le (C+1)\lceil n/k\rceil$ and the loop $\gamma'$ defined by the word $b'_1\dots b'_k$, of length $\le Ck\lceil n/k\rceil$, lying in $\mathcal{F}[k]$. Accordingly
\begin{align*}
\textnormal{area}(\gamma) & \le \sum_{i=1}^k\textnormal{area}(\gamma_i)+\textnormal{area}(\gamma')\\
 & \le k\delta((C+1)\lceil n/k\rceil)+\delta_{\mathcal{F}[k]}(Ck\lceil n/k\rceil).\qedhere
\end{align*}
\end{proof}

\begin{prop}\label{quadra}
Suppose that $\mathcal{F}$ is efficient. Suppose that for some $\zeta>1$, for all $k$, there is a constant $a_k$ such that we have $\delta_{\mathcal{F}[k]}(n)\le a_kn^\zeta$ for all $n$. Then there exists a constant $C'$ such that $\delta(n)\le C'n^\zeta$ for all $n$.
\end{prop}
\begin{proof}
Set $u_k(x)=a_k(2Cx)^\zeta$, where $C$ is given by the definition of efficiency of $\mathcal{F}$. By Lemma \ref{astg}, we have
$$\delta(n)\le u_k(n)+k\delta\left(\frac{(C+1)2n}{k}\right)$$
for all large $n$. The conclusion then follows from Lemma \ref{lemf}.
\end{proof}

\begin{rem}
If $C$ is the constant given in the definition of efficiency, we have the following general inequality
$$\delta_{\mathcal{F}[k]}(n)\le k\delta_{\mathcal{F}[3]}((2C+1)n).$$
Indeed, take a loop in $\mathcal{F}[k]$ of size $n$, defined by a word $u_1\dots u_k$. The endpoint of the path $u_1\dots u_i$ is at distance at most $n/2$ from the origin, and can therefore be joined to the origin by a path $s_i$ in $\mathcal{F}$ of size $\le Cn/2$. The loop $\mu_i$ defined by $s_{i-1}$, $s_{i}$ and $u_i$ lies in $\mathcal{F}[3]$, and the original loop is filled by the $k$ loops $\mu_1,\dots,\mu_k$, yielding the inequality.

In particular, in the hypotheses of Proposition \ref{quadra} it is enough to check the case $k=3$.
\end{rem}


\section{Proof of Theorem \ref{meta}}\label{ProofSection}

Let us turn back to the group $G$ of Theorem \ref{meta}. Refining the decomposition $\bigoplus V_i$ if necessary, we can suppose that each $V_i$ is a local field $\mathbf{K}_i$ on which $A\simeq\mathbf{Z}^d$ acts by scalar multiplication. Fix a multiplicative norm on each $\mathbf{K}_i$, and let $S_i$ be the one-ball in $\mathbf{K}_i$. Let $T$ be a symmetric generating set in $A$, and let $|\cdot|$ denote the word length in $A$ with respect to~$T$. Let $\mathcal{F}$ denote the set of words of the form $$\left(\prod_{i=1}^mt_iv_it_i^{-1}\right)t,$$ where $v_i\in S_i$ and $t_i,t$ are words in the letters of $T$.

\begin{lem}\label{feff}
Let $G=\bigoplus_{i=1}^m\mathbf{K}_i\rtimes A$ be a group as in Theorem \ref{meta}, with each $\mathbf{K}_i$ a local field with an action by scalar multiplication. Replace the last assumption (on pairs $(i,j)$) by the weaker assumption that for every $i$ there exists an element of $A$ acting on $\mathbf{K}_i$ by contractions.  Then $\mathcal{F}$ is efficient.
\end{lem}
\begin{proof} Let $t\in T$ act on $\mathbf{K}_i$ by multiplication by $\lambda_i(t)\in\mathbf{K}_i^*$. Define $c>1$ by $c=\min_i(\max_t|\lambda_i(t)|)$ and $C=\max(2,\max_{t,i}|\lambda_i(t)|)$. 

We first claim that every word of length $\le n$ in the generating set $\bigcup S_i\cup T$ represents an element $xt=((x_i),t)$ of $\bigoplus\mathbf{K}_i\rtimes A$ with 
$\|x_i\|\le C^n$ and $|t|\le n$. This is checked by induction on $n$. If we multiply on the right by an element of $T$, the induction step works trivially. Let us look when we multiply on the right by an element $v$ of $S_i$. Then $xtv=x(tvt^{-1})t$. Then in $\mathbf{K}_i$, $\|tvt^{-1}\|\le C^n$, so $\|x+tvt^{-1}\|\le C^n+C^n\le C^{n+1}$.

Now consider such an element $((x_i),t)$, with $\|x_i\|\le C^n$ and $|t|\le n$ and write it as an element of $\mathcal{F}$ of length $\le Kn$ for some constant $K$. By definition of $c$, we can write $x_i=t_iv_it_i^{-1}$ with $t_i$ a word on the alphabet $T$ of length at most
$$\left\lceil\log_c(C^n)\right\rceil\le n\lceil \log C/\log c\rceil$$
and $\|v_i\|\le 1$ a letter of $S_i$. So the element is represented by the word $(\prod_1^mt_iv_it_i^{-1})t$, which belongs to $\mathcal{F}$ and has length $\le n(2m\lceil\log C/\log c\rceil+1)$.
\end{proof}

\begin{proof}[{\rm \bf Proof of Theorem \ref{meta}}]
The group $G$ has an obvious retraction onto $A\simeq\mathbf{Z}^d$, and therefore its Dehn function is bounded below by that of $\mathbf{Z}^d$, which is linear if $d=1$ and quadratic otherwise.

Let us now prove the upper bound on the Dehn function. As we mentioned above, we can suppose that each $V_i$ is a local field $\mathbf{K}_i$ with action of $A$ by scalar multiplication. Take $S=T\cup\bigcup S_i$ as set of generators. Define $\mathcal{R}$ as the set of relators consisting of
\begin{itemize}
\item finitely many defining relations of $A$ with respect to $T$;
\item relations of length 4 of the form $tst^{-1}=s'$, $t\in T$, $s,s'\in S_i$ for some i;
\item relations of length 4 of the form $[s,s']=1$ when $s\in S_i,s'\in S_j$;
\item relations of length 3 of the form $ss'=s''$ when $s,s',s''\in S_i$.
\end{itemize}
The proof that follows will show that this is a presentation of $G$ and that the corresponding Dehn function is quadratically bounded.

Denote by $\alpha$ the corresponding area function with respect to $\mathcal{R}$. 
Define $$c(m_1,m_2)=\alpha(m_1^{-1}m_2)$$ as the corresponding bi-invariant $[0,\infty]$-valued distance. (The claim that $\mathcal{R}$ is a set of defining relators of $G$ is equivalent to showing that $c$ takes finite values on pairs $(m_1,m_2)$ of homotopic [i.e.\ $m_1\equiv m_2$] words).
It is useful to think of $c$ as the {\it cost} of going from $m_1$ to $m_2$. Recall that we write equality of words as $=$ and equality in $G$ as $\equiv$. By the triangular inequality together with bi-invariance, we get the following useful ``substitution inequality", which we use throughout
$$\alpha(xyz)\le \alpha(xy'z)+c(y,y').$$

\begin{cla}Under the assumptions of Lemma \ref{feff}, there exists a constant $C$ such that for all $i$, if $v,w\in S_i$ and if $s$ is a word of length $\le n$ with respect to $T$ and $svs^{-1}\equiv w$ (i.e.\ $s\cdot v=w$ for the given action), then $c(svs^{-1},w)\le Cn^2$.\label{recon}\end{cla}

Indeed, we can write $s\equiv t$ with $t=t_1t_2$ a word of length $n$ in $T$, where all letters in $t_2$ contract $\mathbf{K}_i$ and all letters in $t_1$ dilate $\mathbf{K}_i$. Clearly, for every right terminal segment $\tau_j$ of $t$ (made of the $j$ last letters in $t$), $\tau v\tau^{-1}$ represents an element $w_j$ of $S_i$. Therefore an immediate induction on the $j$ provides $c(\tau v\tau^{-1},w_j)\le|\tau|$, so $c(tvt^{-1},w)\le n$. Now $A$ has a quadratic Dehn function, so $c(s,t)\le C_1n^2$, hence $c(svs^{-1},tvt^{-1})\le 2C_1n^2$ and thus $$c(svs^{-1},w)\le n+2C_1n^2\le C_2n^2$$ and the claim is proved.

By assumption (of the theorem), there exist elements $s_{ij}$ contracting both $\mathbf{K}_i$ and $\mathbf{K}_j$. We can suppose that all $s_{ij}$ belong to $T$ (enlarging $T$ if necessary). If $s\in S$, let it act on $\mathbf{K}_i$ by multiplication by $\lambda_{s,i}$.  There exists a positive integer $M$ (depending only on $G$ and $T$) such that for all $t\in T$, $s_{ij}^Mt$ contracts both $\mathbf{K}_i$ and $\mathbf{K}_j$.

\begin{cla}\label{coqua}
There exists a constant $C$ such that for all $n$, whenever $t,u$ are words of length $\le n$ with respect to $T$ and $(v,w)\in S_i\times S_j$, we have
$$\alpha([tvt^{-1},uwu^{-1}])\le Cn^2.$$
\end{cla}

From now on, we write $a_n\preceq b_n$ if there exists a constant $C$ (depending only on the group and on $T$) such that for all $n\ge 1$ we have $a_n\le C b_n$. 

To prove Claim \ref{coqua}, set $s=s_{ij}^{Mn}$, so that $st$ contracts $\mathbf{K}_i$ and $su$ contracts $\mathbf{K}_j$.
We have 
$$\alpha([tvt^{-1},uwu^{-1}])=\alpha([stvt^{-1}s^{-1},suwu^{-1}s^{-1}]).$$
We have $stvt^{-1}s^{-1}\equiv v'$ for some $v'\in S_i$ and similarly 
$suwu^{-1}s^{-1}\equiv w'$ for $w'\in S_j$. By Claim \ref{recon}, we have $c(stvt^{-1}s^{-1},v')\preceq n^2$ and $c(suwu^{-1}s^{-1},w')\preceq n^2$. So $$\alpha([tvt^{-1},uwu^{-1}])\le \alpha([v',w'])+2c(stvt^{-1}s^{-1},v')+2c(suwu^{-1}s^{-1},w').$$
(The factor 2 arises since we perform two successive substitutions, as the commutator $[x,y]$ involves two times the letter $x$.) Since $\alpha([v',w'])=1$, we deduce
$$\alpha([tvt^{-1},uwu^{-1}])\preceq n^2.$$
and Claim \ref{coqua} is proved.

Let $\mathcal{F}$ denote the efficient set of words introduced in Lemma \ref{feff}. Fix any integer $k_0$ and let $w$ be a null-homotopic word in $\mathcal{F}[k_0]$, of length at most $n$. It will be convenient to write it as $\prod_{j=1}^{R} t_j v_j$, where $R=mk_0$, $t_j$ is a word in $T$ with $|t_j|\le n$, and $v_j$ is a letter in $\bigcup S_i$. Set $u_j=t_1\dots t_j$, so $|u_j|\le Rn$. Then the word is equal to $$\left(\prod_{j=1}^{R} u_j v_ju_{j}^{-1}\right)u_{R}.$$ Since this word is null-homotopic, $u_{R}\equiv 1$ and $\alpha(u_R)\preceq n^2$. By performing at most $R^2$ times relations of the form $[u_j v_ju_{j}^{-1},u_k v_ku_{k}^{-1}]$, each of which has area $\preceq n^2$ by Claim \ref{coqua}, we rewrite the product above, gathering the terms $u_j v_ju_{j}^{-1}$ for which $v_j$ belongs to the same $S_i$. Namely, we obtain with cost $\preceq R^2n^2\preceq n^2$ a word $$\prod_{i=1}^mw_i,$$ where each $w_i$ is a null-homotopic word of the form $$w_i=\prod_{k=1}^{R_i}u_{j(i,k)}v_{j(i,k)}u_{j(i,k)}^{-1},$$ where $\sum_i R_i=R$, and $v_{j(i,k)}\in S_i$. We are therefore left to prove that $\alpha(w_i)\preceq n^2$ for every $1\leq i\leq m$. 

Set $s=s_{ii}^{MRn}$. Then $su_{j(i,k)}v_{j(i,k)}u_{j(i,k)}^{-1}s^{-1}$ represents an element $x(i,k)$ of $S_i$, and by the Claim \ref{recon}, we have $c(su_{j(i,k)}v_{j(i,k)}u_{j(i,k)}^{-1}s^{-1},x(i,k))\preceq n^2$. This reduces to compute the area of the word $x(i,1)\dots x(i,R_i)$, which is bounded above by $R_i\le R$, i.e.\ by a constant. Therefore the area of $w_i$ is $\preceq n^2$, and accordingly the original word $w$ has area $\preceq n^2$. 

We have just proved that for any $k_0$, there exists a constant $C_{k_0}$ such that for all $n\ge 1$, we have $\delta_{\mathcal{F}[k_0]}(n)\le C_{k_0}n^2$.
By Proposition \ref{quadra} we deduce that the Dehn function of $G$ is at most quadratic.

If $d=1$, since $\mathbf{Z}$ has linear Dehn function, the same proof works to prove that $\delta_{\mathcal{F}[k_0]}$ has linear growth for any $k_0$. However Proposition \ref{quadra} does not apply for $\zeta=1$; we can nevertheless apply it for $\zeta=3/2$, so that the Dehn function of $G$ is asymptotically bounded by $n^{3/2}$, so is subquadratic. A general argument implies that this forces the Dehn function to be linear~\cite{Bowditch}.
\end{proof}


\baselineskip=16pt


\bigskip

\footnotesize

\end{document}